\documentclass[12pt]{amsart}%
\usepackage{amscd,amsmath,latexsym,amsthm,amsfonts,amssymb,graphicx,color,geometry,hyperref}
\usepackage[utf8]{inputenc}
\usepackage{amsmath}
\usepackage{amsfonts}
\usepackage{amssymb}
\usepackage{graphicx}%
\providecommand{\U}[1]{\protect\rule{.1in}{.1in}}
\providecommand{\U}[1]{\protect\rule{.1in}{.1in}}

\newtheorem{theorem}{Theorem}[section]
\newtheorem{proposition}[theorem]{Proposition}

\newtheorem{example}[theorem]{Example}

\newtheorem{lemma}[theorem]{Lemma}

\newtheorem{definition}[theorem]{Definition}
\numberwithin{equation}{section}

\begin{document}
\title[Linear structure in subsets of quasi-Banach sequence spaces ]{Linear structure in certain subsets of quasi-Banach sequence spaces}
\author[Daniel Tomaz]{Daniel Tomaz}
\address{Departamento de Matem\'{a}tica \\
\indent
	Universidade Federal da Para\'{\i}ba \\
\indent
	58.051-900 - Jo\~{a}o Pessoa, Brazil.}
\email{danieltomazmatufpb@gmail.com}
\thanks{The author is supported by Capes}

\begin{abstract}
For $0<p<1,$ we prove that there is a $\mathfrak{c}$-dimensional subspace of $\mathcal{L}\left(  \ell_{p},\ell
_{p}\right)  $ such that, except for the null vector, all of its vectors fail
to be absolutely $(r,s)$-summing regardless of the real numbers $r,s$, with
$1\leq s\leq r<\infty$. This extends a result proved by Maddox in 1987.
Moreover, the result is sharp in the sense that it is not valid for $p\geq1.$

\end{abstract}
\maketitle


\section{Introduction and Notation}

Absolutely summing operators are usually investigated in the setting of Banach
spaces. In 1987, Maddox showed that the definition can be reformulated for
$p$-normed spaces $E$ and $F$ when the topological dual of $E$, denoted by
$E^{\prime},$ is non-trivial; for instance, if $E=\ell_{p}$ with $0<p<1$. For
$1\leq s\leq r<\infty,$ a linear operator $T:E\rightarrow F$ is absolutely
$\left(  r,s\right)  $-summing if $\sum_{k}\left\Vert T\left(  x_{k}\right)
\right\Vert ^{r}<\infty$ whenever $\left(  x_{k}\right)  _{k=1}^{\infty}$ is a
sequence in $E$ such that $\sum_{k}\left\vert f\left(  x_{k}\right)
\right\vert ^{s}<\infty$ for each $f\in E^{\prime}.$ The space of absolutely
$(r,s)$-summing linear operators from $E$ to $F$ will denoted by
$\prod\nolimits_{(r,s)}\left(  E;F\right)  $ and space of bounded linear
operators from $E$ to $F$ will be represented by $\mathcal{L}\left(
E;F\right)  $.

When $0<p<1$ and $x=\left(  x_{k}\right)  _{k=1}^{\infty}\in\ell_{p}$, we
denote the natural $p$-norm of $x$ by
\[
\left\Vert x\right\Vert _{p}=\left(  \sum\limits_{k=1}^{\infty}\left\vert
x_{k}\right\vert ^{p}\right)  ^{\frac{1}{p}}.
\]
Let $E$ be a Banach or quasi-Banach space over $\mathbb{K}=\mathbb{R}$ or
$\mathbb{C}$. A subset $A$ of $E$ is $\mu$-$\text{lineable}$ if $A\cup\left\{
0\right\}  $ contains a $\mu$-$\text{dimensional}$ linear subspace of $E$ and
is called $\mu$-$\text{spaceable}$ if $A\cup\left\{  0\right\}  $ contains a
closed $\mu$-$\text{dimensional}$ linear subspace of $E$. These notions of
lineability and spaceability are due to V.Gurariy (\cite{Gurariy1}), see (\cite{book}). From now
on $\mathfrak{c}=card\left(  \mathbb{R}\right)  $.\newline The purpose of this
paper is to show that the set $\mathcal{L}\left(  \ell_{p},\ell_{p}\right)
\smallsetminus\bigcup\nolimits_{1\leq s\leq r<\infty}\prod\nolimits_{(r,s)}%
\left(  \ell_{p},\ell_{p}\right)  $ is $\mathfrak{c}$-lineable for any $0<p<1$, using a technique of
lineability and spaceability explored in other contexts. In (\cite[Theorem
4.1]{Kitson},) the authors showed that if $E,F$ are Banach spaces, under
certain conditions, the set $\mathcal{K}\left(  E,F\right)  \diagdown
\bigcup\nolimits_{1\leq p<\infty}\prod\nolimits_{p}\left(  E,F\right)  $ is
spaceable, where $\mathcal{K}\left(  E,F\right)  $ denote the ideal compact
linear operators from $E$ to $F$, substantially improving a result obtained in
(\cite[Theorem 2.1]{Pellegrino1}). What about the case of non-locally convex
spaces, like $\ell_{p}$ for $0<p<1$?

It is worth recalling that the structure of quasi-Banach or more generally,
metrizable complete topological vector spaces is quite different from the
structure of Banach spaces. The consequence is that the extension of
lineability/spaceability arguments from Banach to quasi-Banach spaces is not
straightforward in general. \newline In (\cite[page 1]{Maligranda}), the
author points out that many definitions and classical results in Banach spaces
can be translated in a natural way to definitions and results in quasi-Banach
ou $p$-Banach spaces. For example, the standard results depending on Baire
Category, like Open Mapping Theorem, are valid also in the context of
quasi-Banach spaces. However, quasi-normed spaces are not necessarily locally
convex, and the Hahn-Banach theorem and results depending on it are in general
false in this context. For example, the dual space of $L^{p}$ for $0<p<1$ is
$\left\{  0\right\}  $. See, ([\cite{Day}, Theorem 1]).

From now on all Banach and quasi-Banach spaces are considered over a fixed
scalar field $\mathbb{K}$ wich can be either $\mathbb{R}$ or $\mathbb{C}$.

\section{Preliminaries}

Let us split $\mathbb{%
\mathbb{N}
}$ into countably many infinite pairwise disjoint subsets $\left(  \mathbb{%
\mathbb{N}
}_{k}\right)  _{k=1}^{\infty}.$ For each integer $k\in\mathbb{%
\mathbb{N}
}$, write $\mathbb{%
\mathbb{N}
}_{k}=\left\{  n_{1}^{\left(  k\right)  }<n_{2}^{\left(  k\right)  }%
<\cdots\right\}  $ and define%
\[
\ell_{p}^{\left(  k\right)  }:=\left\{  x\in\ell_{p}:x_{j}=0\text{ if }%
j\notin\mathbb{%
\mathbb{N}
}_{k}\right\}  .
\]
In addition, consider the sequence of linear operators $i^{\left(  k\right)
}:\ell_{p}\longrightarrow\ell_{p}^{\left(  k\right)  }$ given by%
\[
i^{\left(  k\right)  }\left(  x\right)  _{a_{i}^{\left(  k\right)  }}=\left\{
\begin{array}
[c]{c}%
u\left(  x\right)  _{i},\text{if }k=j\\
0,\text{ if }k\neq j
\end{array}
\right.  .
\]
Now, for each $k\in\mathbb{%
\mathbb{N}
}$ , consider the sequence $\left(  u_{k}\right)  _{k=1}^{\infty}$ of
operators belonging $\mathcal{L}\left(  \ell_{p};\ell_{p}\right)  $ given by%
\[
u_{k}:\ell_{p}\overset{i^{\left(  k\right)  }}{\longrightarrow}\ell
_{p}^{\left(  k\right)  }\overset{j}{\longrightarrow}\ell_{p}%
\]
where $j:\ell_{p}^{\left(  k\right)  }{\longrightarrow}\ell_{p}$ this is the
inclusion operator.

\begin{lemma}
\label{l1}Let $p>0.$ The operator $T:\ell_{p}\longrightarrow\mathcal{L}\left(
\ell_{p};\ell_{p}\right)  $ given by%
\[
T\left(  \left(  a_{k}\right)  _{k=1}^{\infty}\right)  =\sum\limits_{k=1}%
^{\infty}a_{k}u_{k}%
\]
is well defined, linear and injective.
\end{lemma}

\begin{proof}
For $\left(  a_{k}\right)  _{k=1}^{\infty}\in\ell_{p}$ we have%
\begin{align*}
\sum\limits_{k=1}^{\infty}\left\Vert a_{k}u_{k}\right\Vert _{\mathcal{L}%
\left(  \ell_{p};\ell_{p}\right)  }^{p}  &  =\sum\limits_{k=1}^{\infty
}\left\vert a_{k}\right\vert ^{p}.\left\Vert u_{k}\right\Vert _{\mathcal{L}%
\left(  \ell_{p};\ell_{p}\right)  }^{p}\\
&  =\sum\limits_{k=1}^{\infty}\left\vert a_{k}\right\vert ^{p}.\left\Vert
i^{\left(  k\right)  }\right\Vert _{\mathcal{L}\left(  \ell_{p};\ell
_{p}^{\left(  k\right)  }\right)  }^{p}\\
&  =\sum\limits_{k=1}^{\infty}\left\vert a_{k}\right\vert ^{p}.\left\Vert
u\right\Vert _{\mathcal{L}\left(  \ell_{p};\ell_{p}\right)  }^{p}\\
&  =\left\Vert u\right\Vert _{\mathcal{L}\left(  \ell_{p};\ell_{p}\right)
}^{p}.\sum\limits_{k=1}^{\infty}\left\vert a_{k}\right\vert ^{p}<\infty.
\end{align*}
In this way, the map%
\[
T:\ell_{p}\longrightarrow\mathcal{L}\left(  \ell_{p};\ell_{p}\right)  ,\text{
}T\left(  \left(  a_{k}\right)  _{k=1}^{\infty}\right)  =\sum\limits_{k=1}%
^{\infty}a_{k}u_{k}%
\]
is well-defined. It is clear that $T$ is linear and injective.
\end{proof}

\begin{lemma}
Let $p>0.$ Then $\dim\mathcal{L}\left(  \ell_{p};\ell_{p}\right)
=\mathfrak{c.}$
\end{lemma}

\begin{proof}
For $0<p<1$, $\ell_{p}$ is quasi-Banach space ($p$-Banach-space). So,
$\mathcal{L}\left(  \ell_{p};\ell_{p}\right)  $ is quasi-Banach space (see
(\cite[page 322]{Kalton}). Using the Baire Category Theorem it is not
difficult to prove that every infinite-dimensional quasi-Banach space has
dimension not smaller than $\mathfrak{c}$.$\ $So $\dim\mathcal{L}\left(
\ell_{p};\ell_{p}\right)  \geq\mathfrak{c.}$ On the other hand, let $\gamma$
be a Hamel basis of $\mathcal{L}\left(  \ell_{p};\ell_{p}\right)  $ and
\begin{align*}
f  &  :\gamma\longrightarrow\ell_{p}^{c_{00}\left(
\mathbb{Q}
\right)  }\\
T  &  \longmapsto T_{\left\vert c_{00}\left(
\mathbb{Q}
\right)  \right.  },
\end{align*}
where $\ell_{p}^{c_{00}\left(
\mathbb{Q}
\right)  }$ is the set of all functions from $c_{00}\left(
\mathbb{Q}
\right)  $ to $\ell_{p}.$ By $c_{00}\left(
\mathbb{Q}
\right)  $ we mean the eventually null sequences with rational entries. From
the density of $c_{00}\left(
\mathbb{Q}
\right)  $ to $\ell_{p}$ we conclude that $f$ is injective and so
\[
\dim\mathcal{L}\left(  \ell_{p};\ell_{p}\right)  =card\left(  \gamma\right)
\leq card\left(  \ell_{p}^{c_{00}\left(
\mathbb{Q}
\right)  }\right)  =\mathfrak{c}^{\aleph_{0}}=\mathfrak{c.}%
\]
Therefore, $\dim\mathcal{L}\left(  \ell_{p};\ell_{p}\right)  =\mathfrak{c.}$
The case $p\geq1$ is similar.
\end{proof}

\begin{definition}
Let $p\in(0,\infty).$ A proper subset $\mathcal{D\subsetneq L}\left(  \ell
_{p};\ell_{p}\right)  $ is stable if there is $u\notin\mathcal{D}$ such that
$u_{k}\notin\mathcal{D}$ for all $k$ and the linear operator $T:\ell
_{p}\longrightarrow\mathcal{L}\left(  \ell_{p};\ell_{p}\right)  $ given by%
\[
T\left(  \left(  a_{k}\right)  _{k=1}^{\infty}\right)  =\sum\limits_{k=1}%
^{\infty}a_{k}u_{k}%
\]
is such that $T\left(  \left(  a_{j}\right)  _{j=1}^{\infty}\right)
\subset\mathcal{L}\left(  \ell_{p};\ell_{p}\right)  \setminus\mathcal{D}$, for
all $0\neq\left(  a_{j}\right)  _{j=1}^{\infty}$ in $\ell_{p}$.
\end{definition}

\begin{proposition}
\label{p1} Let $p>0$ and $\mathcal{D\subsetneq L}\left(  \ell_{p};\ell
_{p}\right)  $ be stable. Then $\mathcal{L}\left(  \ell_{p};\ell_{p}\right)
\diagdown\mathcal{D}$ is $\mathfrak{c}$-lineable.
\end{proposition}

\begin{proof}
Since $T$ is linear and injective, $T\left(  \ell_{p}\right)  $ is a
$\mathfrak{c}$-dimensional subspace of $\mathcal{L}\left(  \ell_{p};\ell
_{p}\right)  .$ Since $\mathcal{D}$ is stable we have $T\left(  \ell
_{p}\right)  \diagdown\left\{  0\right\}  \subset\mathcal{L}\left(  \ell
_{p};\ell_{p}\right)  \diagdown\mathcal{D}$. Therefore, $\mathcal{L}\left(
\ell_{p};\ell_{p}\right)  \diagdown\mathcal{D}$ is $\mathfrak{c}$-lineable.
\end{proof}

The next theorem is fundamental to the proof of our main result. It has been
proven by Maddox in (\cite{Maddox}), from results obtained by Macphail in
(\cite{Macphail}).

\begin{theorem}
\label{t1} (\cite[Theorem 4]{Maddox}) Let $0<p<1$ and $1\leq s\leq r<\infty$.
Then the identity map $i:\ell_{p}\longrightarrow\ell_{p}$ is non $\left(
r,s\right)  $-absolutely summing.
\end{theorem}

\section{The main result}

\begin{theorem}
\label{mt} Let $0<p<1$ and $1\leq s\leq r<\infty$. Then $\mathcal{L}\left(
\ell_{p},\ell_{p}\right)  \smallsetminus\bigcup\nolimits_{1\leq s\leq
r<\infty}\prod\nolimits_{(r,s)}\left(  \ell_{p},\ell_{p}\right)  $ is
$\mathfrak{c}$-lineable. Moreover, the result is sharp, since it is not valid
for $p\geq1.$
\end{theorem}

\begin{proof}
It is enough to prove that $\mathcal{D}:=\bigcup\nolimits_{1\leq s\leq
r<\infty}\prod\nolimits_{(r,s)}\left(  \ell_{p},\ell_{p}\right)  $ is stable.
Let $i:\ell_{p}\longrightarrow\ell_{p}$ the identity map. By Theorem \ref{t1},
$i$ is non $\left(  r,s\right)  $-absolutely summing for $1\leq s\leq
r<\infty$ , that is, $i\notin\mathcal{D}$. Now, let us split $\mathbb{%
\mathbb{N}
}$ into countably many infinite pairwise disjoint subsets $\left(  \mathbb{%
\mathbb{N}
}_{k}\right)  _{k=1}^{\infty}.$ For each integer $k\in\mathbb{%
\mathbb{N}
}$, write $\mathbb{%
\mathbb{N}
}_{k}=\left\{  n_{1}^{\left(  k\right)  }<n_{2}^{\left(  k\right)  }%
<\cdots\right\}  $ and define%
\[
\ell_{p}^{\left(  k\right)  }:=\left\{  x\in\ell_{p}:x_{j}=0\text{ if }%
j\notin\mathbb{%
\mathbb{N}
}_{k}\right\}  .
\]
In addition, consider the sequence of linear operators $i^{\left(  k\right)
}:\ell_{p}\longrightarrow\ell_{p}^{\left(  k\right)  }$ given by%
\[
i^{\left(  k\right)  }\left(  x\right)  _{a_{i}^{\left(  k\right)  }}=\left\{
\begin{array}
[c]{c}%
i\left(  x\right)  _{i},\text{if }k=j\\
0,\text{if }k\neq j
\end{array}
\right.  .
\]
Now, for each $k\in\mathbb{%
\mathbb{N}
}$ , consider the sequence $\left(  u_{k}\right)  _{k=1}^{\infty}$ of
operators belonging $\mathcal{L}\left(  \ell_{p};\ell_{p}\right)  $ given by%
\[
u_{k}:\ell_{p}\overset{i^{\left(  k\right)  }}{\longrightarrow}\ell
_{p}^{\left(  k\right)  }\overset{j}{\longrightarrow}\ell_{p}%
\]
where $j:\ell_{p}^{\left(  k\right)  }{\longrightarrow}\ell_{p}$ this is the
inclusion operator. Note that
\begin{center}
$\left\Vert u_{k}\right\Vert _{\mathcal{L}\left(  \ell_{p};\ell_{p}\right)
}=\left\Vert i^{\left(  k\right)  }\right\Vert _{\mathcal{L} \left(  \ell
_{p};\ell_{p}^{\left(  k\right)  }\right)  }=\left\Vert i\right\Vert
_{\mathcal{L}\left(  \ell_{p};\ell_{p}\right)  }$
\end{center}

for each positive integer $k$. Like this, $u_{k}\notin\mathcal{D}$. For
$0\neq\left(  a_{k}\right)  _{k=1}^{\infty}$ in $\ell_{p}$, by Lemma \ref{l1},
the operator $T:\ell_{p}\longrightarrow\mathcal{L}\left(  \ell_{p};\ell
_{p}\right)  $ given by%
\[
T\left(  \left(  a_{k}\right)  _{k=1}^{\infty}\right)  =\sum\limits_{k=1}%
^{\infty}a_{k}u_{k}%
\]
is well defined, linear and injective. Since the supports of the operators
$u_{k}$ are pairwise disjoints, we have $T\left(  \left(  a_{k}\right)
_{k=1}^{\infty}\right)  \subset\mathcal{L}\left(  \ell_{p};\ell_{p}\right)
\setminus\mathcal{D}$. Therefore, $\mathcal{D}$ is stable. The result follows
then from the Proposition \ref{p1}.

For $p\geq1$, it is well known that%
\[
\prod\nolimits_{(\max\{p,2\},1)}\left(  \ell_{p},\ell_{p}\right)
=\mathcal{L}\left(  \ell_{p};\ell_{p}\right)
\]
and thus%
\[
\mathcal{L}\left(  \ell_{p};\ell_{p}\right)  \setminus\mathcal{D}=\emptyset.
\]

\end{proof}

Our main result shows that, in the case of absolutely summing operators, we
have a different results when dealing with quasi-Banach and Banach spaces.
However, in other settings quasi-Banach and Banach spaces behave similarly, as
we can see in the following examples:

\begin{example}
\label{c1} For $p,q>0$, the set of injective operators in $\mathcal{L}\left(
\ell_{p};\ell_{q}\right)  $ is $\mathfrak{c}$-lineable. In fact, defining
$\mathcal{D}$ as the set of non-injective operators in $\mathcal{L}\left(
\ell_{p};\ell_{q}\right)  $, we can easily see that $\mathcal{D}$ is stable
and by Proposition \ref{p1} we conclude that $\mathcal{L}\left(  \ell_{p}%
;\ell_{q}\right)  \diagdown\mathcal{D}$ is $\mathfrak{c}$-lineable.
\end{example}

Using a variation of the notion of stability we can also easily obtain the
following examples:

\begin{example}
	For $p,q>0$ the set $\mathcal{D=}\left\{  T\in\mathcal{L}\left(  \ell_{p}%
	;\ell_{q}\right)  :\text{ }T\text{ is non surjective}\right\}  $ is spaceable.
	In fact, for each integer $k$, consider the sequence of operators $T_{k}%
	:\ell_{p}\longrightarrow\ell_{q}$ in $\mathcal{L}\left(  \ell_{p};\ell
	_{q}\right)  $ given by%
	\[
	T_{k}\left(  x\right)  =\left(  0,0,\ldots,x_{k},0,0,\ldots\right)
	\]
	with $x_{k}$ in $\left(  k+1\right)  $th position and $x$ in $\ell_{p}$. Note
	that $T_{k}$ is non surjective for all $k\in%
	\mathbb{N}
	.$ Moreover, the operator $\Psi:\ell_{1}\longrightarrow\mathcal{L}\left(
	\ell_{p};\ell_{q}\right)  $ given by $\Psi\left(  \left(  a_{k}\right)
	_{k=1}^{\infty}\right)  =\sum\limits_{k=1}^{\infty}a_{k}T_{k}$ is well defined
	because,%
	\begin{align*}
	\sum\limits_{k=1}^{\infty}\left\Vert a_{k}T_{k}\right\Vert  & =\sum
	\limits_{k=1}^{\infty}\left\vert a_{k}\right\vert .\left\Vert T_{k}\right\Vert
	\\
	& =\sum\limits_{k=1}^{\infty}\left\vert a_{k}\right\vert .1<\infty.
	\end{align*}
	It is clear that $\Psi$ is linear and injective. By construction of operators
	$T_{k}$, $\overline{\Psi\left(  \left(  a_{k}\right)  _{k=1}^{\infty}\right)
	}$ is non surjective. Therefore, $\overline{\Psi\left(  \left(  a_{k}\right)
	_{k=1}^{\infty}\right)  }\subset\mathcal{D}$. We conclude that $\mathcal{D}$
is spaceable.
\end{example}

\begin{example}
	For $p,q>0$ the set $\mathcal{D=}\left\{  T\in\mathcal{L}\left(  \ell_{p}%
	;\ell_{q}\right)  :\text{ }T\text{ is non injective}\right\}  $ is spaceable.
	In fact, the proof is analogous to that of the previous example, it is
	sufficient to consider the sequence of operators $R_{k}:\ell_{p}%
	\longrightarrow\ell_{q}$ in $\mathcal{L}\left(  \ell_{p};\ell_{q}\right)  $
	given by%
	\[
	R_{1}\left(  x\right)  =\left(  x_{2},0,0,\ldots\right)  \text{ and }%
	R_{k}\left(  x\right)  =\left(  0,0,0,\ldots x_{k+1},0,\ldots\right)  \text{
		for }k>1
	\]
	with $x_{k+1}$ in $k$th position. We can easily see that $R_{k}$ is non
	injective for all $k\in%
	\mathbb{N}
	$ and the result follows.
\end{example}
\section*{Acknowledgments}
This Paper is part of the PhD Thesis of the author under supervision of Professor Nacib Albuquerque.

\end{document}